\documentclass[a4paper,leqno,11pt]{amsart}

\usepackage{amsfonts,amssymb,verbatim,amsmath,amsthm,latexsym,textcomp,amscd}
\usepackage{latexsym,amsfonts,amssymb,epsfig,verbatim}
\usepackage{amsmath,amsthm,amssymb,latexsym,graphics,textcomp}
\usepackage{graphicx}
\usepackage{color}
\usepackage{url}
\usepackage{enumerate}
\usepackage[mathscr]{euscript}
\input xy
\xyoption{all}

\usepackage{hyperref}

\newcommand{\Z}{\mathbb Z}

\newcommand{\MM}{\mathcal{M}}
\newcommand{\DD}{\mathcal{D}}
\newcommand{\KK}{\mathcal{K}}
\newcommand{\OO}{\mathcal{O}}

\newcommand{\R}{\mathbb R}
\newcommand{\C}{\mathbb C}

\newcommand{\Map}{\mathit{Map}}
\newcommand{\Hom}{\mathit{Hom}}

\theoremstyle{plain}
\newtheorem{theorem}{Theorem}[section]
\newtheorem{thm}[theorem]{Theorem}
\newtheorem{lemma}[theorem]{Lemma}
\newtheorem{cor}[theorem]{Corollary}
\newtheorem{prop}[theorem]{Proposition}

\theoremstyle{definition}
\newtheorem{defn}[theorem]{Definition}

\newtheorem{exam}[theorem]{Example}

\author{Samik Basu and Surojit Ghosh}

\email{samik.basu2@gmail.com; samik@rkmvu.ac.in  }
\address{Department of Mathematics,
Vivekananda University,
Belur, Howrah - 711202, West Bengal, India.}

\email{surojitghosh89@gmail.com}
\address{Department of Mathematics,
Vivekananda University,
Belur, Howrah - 711202, West Bengal, India.} 

 
\begin{document}

\title{ Equivariant maps related to the topological Tverberg conjecture}
\maketitle

\begin{abstract}
Using equivariant obstruction theory we construct equivariant maps from certain universal spaces to representation spheres for cyclic groups, products of elementary Abelian groups and dihedral groups. Restricting them to finite skeleta constructs equivariant maps between spaces which are related to the topological Tverberg conjecture. This answers negatively a question of \"Ozaydin posed in relation to weaker versions of the same conjecture. Further, it also has consequences for Borsuk-Ulam properties of representations of cyclic and dihedral groups. 
\end{abstract}

\section{Introduction}
This paper deals with the construction of equivariant maps from certain infinite dimensional complexes into representation spheres. These maps bear a curious connection to equivariant maps related to the Topological Tverberg conjecture (\cite{B-B79, B-S-S}). The conjecture states that for every map $f$ from the simplex $\Delta^{(d+1)(n-1)}$ to $\R^d$, there are $n$ disjoint faces whose images have a common intersection point. This turned out to be a significant unsolved problem in combinatorics having a positive answer for prime powers $n$ (\cite{B-S-S},\cite{Oz87}). Very recently counter-examples have been discovered when $n$ is not a prime power in most cases in \cite{fr15} (more precisely, when $d\geq 3n+1$ and $n$ is not a prime power). The lowest counter-examples up to now have been constructed in \cite{A-M-S-W15}(when $d\ge 2n$ and $n$ is not a prime power).

 The proof of the Topological Tverberg theorem in the prime power case involves a reduction to a question in equivariant homotopy theory. A map violating the topological Tverberg conjecture leads to a map from the $n$-fold deleted join of $\Delta^N$ ($N=(d+1)(n-1)$) to the complement of a diagonal subspace in the $n$-fold join of $\R^d$. Moreover, this map commutes with the action of $\Sigma_n$, the permutation group on $n$ elements. After making certain identifications one obtains a $\Sigma_n$-map from the $N+1$-fold join 
$\{1,\cdots,n\}^{\ast N+1}$  to $S(W^d)$  where $W$ denotes the standard representation of $\Sigma_n$ and the notation $S(W^d)$ denotes the sphere in $W^d$ (This reduction is clearly explained in \cite{mat08}). Non-existence of such equivariant maps imply the topological Tverberg conjecture. The case $n=2$ of the above problem reduces to the question of $C_2$-equivariant maps from $S^{N+1}\to S^N$ which do not exist by the Borsuk-Ulam theorem. 


The first proof of the topological Tverberg conjecture for the prime power case appeared in \"Ozaydin (\cite{Oz87}). He did this by proving the restriction of the above equivariant question to $(\Z/p)^r \subset \Sigma_{p^r}$ does not have a solution. In \cite{sar}, Sarkaria considers Borsuk-Ulam properties of representations inspired by the restriction of the equivariant question to groups $G \subset \Sigma_n$ and also provides a new proof in the prime power case. These may be viewed as generalizations of the Borsuk-Ulam theorem.

In \cite{Oz87}, \"Ozaydin further observes that when $n$ is not prime, the $C_n$-equivariant maps  $\{1,\cdots,n\}^{\ast N+1} \to S(\bar{\rho}^d)$  do exist for $N= (d+1)(n-1)$ (Observe that the standard representation $W$ of $\Sigma_n$ restricts to the reduced regular representation $\bar{\rho}$ of $C_n$). Using equivariant obstruction theory he constructs maps from free $(d+1)$-dimensional $G$-complexes to $(d-1)$-connected ones under some hypothesis. The statement for the symmetric group is used by Frick to construct the counterexamples in \cite{fr15} building on the work of Mabillard and Wagner (\cite{M-W15}). With respect to the counter-examples, a natural question to ask will be : Does a weaker version of the Tverberg conjecture hold when $n$ is not a prime power? More specifically, by weaker version of Tverberg conjecture, we mean choosing a greater value of $N$--\\
{\it  For a map $f : \Delta^{N} \to \R^d$ (with $N > (d+1)(n-1)$), are there $n$ disjoint faces whose images have a common intersection point? }\\
One readily notes that the equivariant obstruction theory arguments are dimension-sensitive and leaves open the possibility of proving weak Tverberg theorems as above. Recently, certain upper and lower bounds on $N$ have been found in \cite{B-F-Z15}. As in the Tverberg case, a map violating the above hypothesis induces a $\Sigma_n$-equivariant map $\{1,\cdots, n\}^{\ast N+1} \to S(W^d)$. This leads to the following question-- \\
{\it For a subgroup $G\subset \Sigma_n$, does there exist a $G$-equivariant map $\{1,\cdots, n\}^{\ast N+1} \to S(W^d)$ with $N > (d+1)(n-1)$?}\\
\"Ozaydin introduces this question and particularly instructs the reader to try for obstructions to such maps using equivariant $K$-theory especially for the cyclic group case (\cite{Oz87}-- paragraph following Proposition 5.1). In this paper, we consider the above question for $G=C_n$ the cyclic group, $G=$ a product of elementary Abelian groups, and $G=D_n$ the dihedral group (for $n$ odd). We prove that such equivariant maps exist in this case. 

For the $C_n$ case, the left hand space is a free $C_n$-space and so it sits inside the universal space $EC_n$. Thus, one is led to the question of existence of $C_n$-equivariant maps from  $EC_n$ to $S(\bar{\rho}^d)$. Maps out of  $EG$ to a $G$-space $X$ correspond to homotopy fixed points of $X$, and their existence can be studied using the Sullivan conjecture (\cite{May},\cite{Mill}). If $G$ is a $p$-group and $X$ is a finite $G$-space without fixed points, then Sullivan's conjecture states that after $p$-completion, $X^G$ is equivalent to $X^{hG}$. Therefore, if the fixed points are empty so are the maps out of $EG$. However, if $G$ is not a $p$-group there do exist finite complexes $X$ with $X^G=\varnothing$ which support equivariant maps $EG\to X$ (\cite{Dror}). For $G=C_n$ with $n\neq p^k$ for any $p$ or $k$,  one may proceed by equivariant obstruction theory to show that a $C_n$-space $X$ with $X^{C_n}=\varnothing$ maps to $S(\bar{\rho}^d)$ for $d \gg 0$. 

In fact, one may do a lot better. We follow a construction in \cite{Dror} to deduce that there are equivariant maps $EC_n \to S(\bar{\rho}^d)$ for any $d$ when $n$ is not a prime power (cf. Corollary \ref{tver}) resolving the question in the cyclic group case. In the case of the dihedral group $D_n \subset \Sigma_n$, we work out a more delicate obstruction theory computation to make the same conclusion. This rules out proofs of weaker versions of Tverberg's theorems using the groups $C_n$ and $D_n$ (cf. Corollary \ref{tver2}). 

The techniques of the paper also have applications to Borsuk-Ulam-like properties of the representations of $C_n$ and $D_n$. There are many definitions of Borsuk-Ulam properties in the literature. We consider Borsuk-Ulam properties for specific representations as considered in \cite{sar} and \cite{iz-ma}. The first refers to constructing equivariant maps out of skeleta of $EG$ and the second to constructing maps out of representation spheres of high dimension. These turn out to be analogous when $G$ is a cyclic group. 

We define that a fixed point free representation $V$ has the anti-Borsuk-Ulam property if there are maps from representation spheres of arbitrarily large dimension to $S(V)$. Our techniques imply that if $n$ is square free, there is a dichotomy among fixed point free $C_n$-representations $V$ -- either $V$ has the Borsuk-Ulam property, or $V$ has the anti-Borsuk-Ulam property (cf. Theorem \ref{borcycsqfree}). When $n$ is not square free, Sullivan's conjecture restricts the existence of equivariant maps from $EC_n$ to $S(V)$ if $S(V)^{C_{p^k}}=\varnothing$ for $p^k \mid n$. We observe that if this condition is avoided by $V$, it satisfies the anti-Borsuk-Ulam property (cf. Theorem \ref{borcyc}). Analogous results are also obtained for the group $D_n$ where $n$ is odd and not a prime power (cf. Theorem \ref{bordnsqfree}, \ref{bordn}). 

This paper is organised as follows. In section 2, we introduce some preliminaries on equivariant obstruction theory. In section 3, we discuss the results on the existence of maps out of universal spaces. In section 4, we dicuss some applications of the results of section 3 specially in the context of Borsuk-Ulam theorems. 

\mbox{ }\\

\noindent

\section{Prelimaniries}

In this section we recall some basic definitions, notations and results in equivariant obstruction theory. Most of these ideas are from  \cite{May} and \cite{bredon}. In equivariant homotopy theory, one uses orthogonal representations to construct $G$-spheres and disks.
\begin{defn}For an orthogonal representation V, the unit disc in the
representation is defined as the space  $$D(V)= \{ v \in V :\parallel v \parallel \leq 1\}$$ and the unit sphere is $$S(V)= \{ v \in V :\parallel v \parallel = 1\}.$$ The spaces $D(V )$ and $S(V )$ have an induced $G$-action and there is an equivariant homeomorphism $$D(V)/S(V) \cong S^V$$ where $S^V$ is the one point compactification of $V$.
\end{defn}

The equivariant homotopy category is defined using weak equivalences on fixed points.
\begin{defn} 

A $G$-map $f:X\rightarrow Y$ between pointed $G$-spaces is called $G$-weak equivalence if $f^H : X^H \rightarrow Y^H$ is a weak equivalence for all subgroups $H \leq G.$ For $n \in \mathbb{N}$  and $H \leq G$,
define $$\pi_n ^H(X) = \left[ G/H_+ \wedge S^n , X \right]^G_{\ast} \cong \pi_n(X^H). $$
\end{defn}

This definition is accompanied by the concept of equivariant CW-complexes which are filtered $G$-spaces obtained by attaching $G$-cells of increasing dimension. Note that a $G$-action on a disjoint union of points is classified upto $G$-equivalence by the decomposition into orbits, and therefore, the $0$-cells are a disjoint union of copies of $G/H$ for conjugacy classes of subgroups $H$. Equivariant $n$-cells are the $G$-spaces $G/H\times D^n$(with trivial action on $D^n$) where $H \subseteq G$ runs over conjugacy classes of subgroups. The boundary of such a cell is $G/H \times S^{n-1}.$ One defines
\begin{defn} A relative $G$-CW complex is a pair of $G$-spaces $(X,A)$  together with a filtration $\{X^{(n)} \}$ of $X$  such that \\
(a) $A \subset X^{(0)}$.\\
(b) $X^{(n+1)}$ is obtained by $X^{(n)}$ as a pushout
$$\xymatrix{ \amalg_{j \in J}G/H_j \times S^n  \ar[d] \ar[r] &X^{(n)} \ar[d]\\ \amalg_{j \in J}G/H_j \times D^{n+1} \ar[r] & X^{(n+1)}}$$\\
(c)  $ X = \cup_{n \in \mathbb{Z}}X^{(n)}$ has the colimit topology.

\end{defn}

One may note that the disks $D(V)$, and the spheres $S^V$ are examples of $G$-CW-complexes. There is an equivariant Whitehead theorem (Theorem 3.4 of  \cite{wan80}) which states that a weak equivalence between $G$-CW-complexes is a $G$-homotopy equivalence. 

 In this paper, we are interested in the construction of equivariant maps using obstruction theory. For equivariant maps, the obstruction theory can be carried out as usual and the result produces obstruction classes in Bredon cohomology (with local coefficients). We briefly recall the theory here.

Let $\OO_G$ denote the orbit category of $G$ whose objects are orbit spaces $G/H$ and morphisms are equivariant maps.  

\begin{defn} A coefficient system is a contravariant functor from $\OO_G$ to  Abelian groups. 
\end{defn}

An example of a coefficient system is the assignment $\underline{\pi_n(X)}(G/H)= \pi_n(X^H)\cong [G/H_+ \wedge S^n,X]_{\ast}^G$ associated to any $G$-space $X$. Taking singular chains is another example. Denote
by $\underline{C}_n(X)$ the coefficient system 
$$\underline{C}_n(X)(G/H) = C_n(X^H;\mathbb{Z})$$
the singular $n$-chains of the space $X^H.$ The boundary map $\partial_n: \underline{C}_n(X) \rightarrow \underline{C}_{n-1}(X)$ is the natural transformation induced by the boundary maps of the singular chain complex $\partial_n(G/H): C_n(X^H; \mathbb{Z}) \rightarrow C_{n-1}(X^H; \mathbb{Z})$. 

Let $\MM$ be a coefficient system. Define the cochain complex of $X$ with respect $\MM$ as 
$$C^n_G(X;\MM) = \Hom _{\mathcal{C}}(\underline{C}_n(X);\MM)$$
with coboundary $\delta :C^n_G(X;\MM) \to C^{n+1}_G(X;\MM)$ defined as $\Hom_{\mathcal{C}}(\partial,id).$ 

\begin{defn}
The Bredon cohomology of $X$ with coefficients in $\MM$ is defined as the cohomology of the cochain complex $C^*_G (X;\MM )$. It is denoted $H^*_G (X;\MM )$.

\end{defn} 

The obstruction theory for extending maps from  CW-complexes can be formulated using Bredon cohomology in a similar way as the non-equivariant situation. Let $X$ be a $G$-CW complex and  let $Y$ be a $G$-space such that $Y^H$ is non-empty, path connected and simple (that is, the fundamental group acts trivially on the higher homotopy groups) for every subgroup $H$ of $G$. Then, given an equivariant map $ \phi : X^{(n)} \rightarrow Y,$ one may define an equivariant obstruction cocycle $c_{G}(\phi) \in C_{G}^{n+1}(X; \underline{\pi_n(Y)})$ leading to the following Proposition. (see Theorem 1.5 of Chapter II in \cite{bredon}.)

\begin{prop}
Let  $\phi : X^{(n)} \rightarrow Y$ be a $G$-map. Then, the restriction of $\phi$ to $X^{(n-1)}$ can be extended to $X^{(n+1)} $ if and only if $[c_{G}(\phi)]\in H^{n+1}_G(X;\underline{\pi_n(Y)})$ is $ 0.$
\end{prop}

For our purposes we need to lift maps from the base to the total space of a $G$-fibration. This theory may be read off from \cite{Mol90} or \cite{Muk94}. Consider a lifting diagram 
$$\xymatrix{X\ar[r]^f  &  Y \\
                                 &  Z \ar[u]^u \ar@{-->}[lu] }$$
Assume that the equivariant homotopy fibre of $f$, $F(f)$ satisfies that for every subgroup $H\subset G$, $F(f)^H$ is a simple space. Then, the homotopy groups $\underline{\pi_nF(f)}$ induces an equivariant local coefficient system on the base $Y$ and by pullback $u^*  \underline{\pi_nF(f)}$ on $Z$. 

Filter $Z$ by skeleta as $Z=\cup_n Z^{(n)}$. Suppose we have a lift $\phi: Z^{(n)}\to X$ for $n\geq 1$. Then one may define obstruction classes $[c_u(\phi)]\in H^{n+1}_G(Z;  u^*  \underline{\pi_nF(f)})$ so that  
\begin{prop}\label{eqobs}
The restriction of the lift $\phi$ to $Z^{(n-1)}$ can be extended to $Z^{(n+1)}$ if and only if $[c_u(\phi)]\in H^{n+1}_G(Z;u^*\underline{\pi_nF(f)})$ is $ 0.$
\end{prop}

Our interest throughout the document is the case : 
$$\begin{cases} G=\Gamma/H\\
Y=\Map(EG,W)_v\\
X= \Map^H(EG,W)_v\\
f:X\to Y=\text{the standard inclusion}
\end{cases}$$
In the above, the subscript $v$ refers to the connected component of $v$,  $W$ is usually a $\Gamma$-representation sphere (hence, a simple space) on which $H$ acts via maps of degree $1$ and $Z$ is some universal space for a family of subgroups, and $u$ is some constant map. For these cases we note the following : 
\begin{prop}\label{fibhtpy}
The fibre $F(f)$ is path-connected and has simple fixed points for every subgroup $K$ of $G$. 
\end{prop}

\begin{proof}
This follows from the proof of Lemma 3.3 of \cite{Dror}.
\end{proof}

\begin{prop}
The local coefficients $u^*  \underline{\pi_nF(f)}$ reduces to the coefficient system $\underline{\pi_nF(f)}$. 
\end{prop}

\begin{proof}
This follows either because $u$ is some constant map or because $Z$ is some universal space, hence, there does not exist any non-trivial coefficient system on $Z$. 
\end{proof}

\section{Construction of equivariant maps from universal spaces }
This section deals with the construction of equivariant maps from universal spaces and universal spaces of subgroups which turn up in the topological Tverberg theorem. Recall that a map $\Delta^N \to \R^d$ which violates the Tverberg's condition induces a $\Sigma_n$-equivariant map 
\begin{equation}\label{Tv}
\{1,\cdots,n\}^{\ast N+1} \to S(W^d)
\end{equation}
where $W$ denotes the standard representation of $\Sigma_n$. We consider the subgroups $G=C_n$ and $G=D_n$ (for $n$ odd) of $\Sigma_n$ and show that there does exist $G$-equivariant maps as above. This is proved by constructing equivariant maps exist out of appropriate universal spaces. Therefore, in order to prove weak versions of Tverberg's theorem along these lines, one needs to consider other subgroups of $\Sigma_n$.

For techniques of constructing equivariant maps from universal spaces we use the ideas of \cite{Dror}. In that paper it was proved that if $G$ is not a $p$-group then there is a fixed point free finite complex $W$ which carries an equivariant map from $EG$ onto it. In the case $G$ is a semi-direct product of a $p$-group and a $q$-group for primes $p\neq q$, there is an explicit construction of $W$. For our purposes we adapt the proof for $G=C_n$ and $G=D_n$ to deduce results about the existence of $G$-maps of the type (\ref{Tv})  in large dimensions. 

We use the following version of Lemma 3.3 of \cite{Dror}. Let $W$ be a simply connected $C_m$-space with $W^{C_m}\neq \varnothing$  and for each $g\in C_m$ the induced map on $W$ is homotopic to the identity. Then,
\begin{lemma}\label{fibm}
For any contractible space $E$ with free $C_m$ action, define $F_i$ to be the homotopy fibre of the map $i : \Map^{C_m}(E, W)_u \rightarrow \Map(E, W)$ (the space $\Map^{C_m}(E, W)_u$ refers to a path component of a $C_m$-equivariant map $u : E \to W$). Then $F_i$ is path connected and its homotopy groups are $m$-profinite groups. (That is, inverse limit of groups with $p$-torsion for $p|m$)
\end{lemma}

\begin{proof}
The proof of Lemma 3.3 in \cite{Dror} may be repeated verbatim together with the observation that the group cohomology of $C_m$ has only $m$-torsion in positive degrees. 
\end{proof}

\subsection{Equivariant maps for cyclic groups}\label{tv}

Using Lemma \ref{fibm}, we construct $C_n$-equivariant maps between universal spaces and representation spheres. Recall that the set of real irreducible representations of $C_n$ are listed as 
$$\begin{cases}\{ 1, \xi, \xi^2 , \cdots ,\xi^\frac{n-1}{2} \}, \text{if $n$ is odd.} \\ 
\{ 1,\sigma ,\xi, \xi^2 , \cdots ,\xi^{\frac{n}{2}-1} \}, \text{if $n$ is even.} \end{cases}$$
The notation above can be read off as follows. For a cyclic group $C_n$ of odd order complex irreducible representations are one-dimensional and given by multiplication by roots of unity; the real irreducible representations are obtained by restriction which are two dimensional. Denote by $\xi$ the representation obtained by multiplication by $e^{2\pi i/n}$. Recalling that the restriction of a complex representation satisfies $res(V)\cong res(\bar{V})$, the set of irreducible representations of $C_n$ are 
 $\{ 1, \xi^r, 1\leq r \leq \frac{n-1}{2} \}$.

For the group $C_{2n}$ there is a $1$-dimensional irreducible real representation given by sign. The other representations are given by the $n^{th}$ roots of unity. Thus the irreducible representations are $\{ 1,\sigma ,\xi, \xi^2 , \cdots ,\xi^{n-1} \}$. This justifies the notation above.

Let $G=C_n$ such that $n$ is not a prime power. That is, $n$ is divisible by at least two distinct primes. Write $n = p^k\cdot m $ so that $k\geq 1, m>1$ and $p$ does not divide $m$. Hence $C_n = C_{p^k} \times C_m.$

 Equivariantly embed $C_m $ inside $ S(\xi^{p^k})$  and  $C_n$ inside $S(\xi)$. Define 
$$W = S(\xi^{p^k})\ast (S(\xi)/C_m) \cong S(\xi^{p^k})\ast S(\xi^m) = S(\xi^{p^k}\oplus \xi^m)$$ 
Then $W$ is  a fixed point free $C_n$-space satisfying the conditions of Lemma \ref{fibm}.




\begin{prop}\label{eqEG}

For $n$ is not a prime power, there is a $C_n$-equivariant map from $EC_n$ to $W.$

\end{prop}
 \begin{proof}
We follow the proof of Proposition $3.3$ in \cite{Dror}. Take a point $x \in S(\xi^{p^k}) \subset W$ and denote by $x$ the corresponding constant map $EC_n \rightarrow W$. Since $ W ^{C_{p^k}} = S(\xi^{p^k})$ then $x$ is $C_{p^k}$-equivariant and hence $x \in \Map^{C_{p^k}}(EC_n , W).$
Consider the lifting diagram of $C_m$-spaces
$$\xymatrix{\Map^{C_{p^k}}(EC_n , W)\ar[r]  &  \Map(EC_n , W) \\
                                                                            &  EC_m \ar[u]^u \ar@{-->}[lu] }$$
As $W^{C_m} \neq \varnothing , $ there exists a $C_m$-equivariant constant map $ u : EC_m \rightarrow \Map(EC_n , W)$. Then we apply equivariant obstruction theory in Proposition \ref{eqobs}, to observe that the obstructions lie in certain groups of the form $H^k(BC_m; M_i)$ for $k\geq 1$ where the modules $M_i$ are $p$-profinite as abelian groups by Lemma \ref{fibm}. These groups are indeed $0$ as $p\nmid m$. Therefore the lift exists. 

Now, we apply the exponential correspondence which says that a $G$-equivariant map $X\times Y\to Z$ is equivalent to a $G$-equivariant map $X\to \Map(Y,Z)$ (Remark 1.1 of \cite{Dror}).  Using this, a $C_m$-equivariant map $EC_m \to \Map^{C_{p^k}}(EC_n , W)$ induces a $C_n$-equivariant map $EC_m\times EC_n \to W$ via the following identifications
\begin{align*}
\Map^{C_m}(EC_m,  \Map^{C_{p^k}}(EC_n , W)) &\simeq  \Map^{C_n}(EC_m,  \Map^{C_{p^k}}(EC_n , W)) \\
                                                                                 & \simeq \Map^{C_n}(EC_m,  \Map(EC_n , W)) \\
                                                                                 & \simeq \Map^{C_n}(EC_m\times EC_n , W))
\end{align*}
In the above, the first equivalence stems from the fact that the action of $C_{p^k}$ on $EC_m$ and  $\Map^{C_{p^k}}(EC_n , W)$ are trivial and the second from the fact that a $C_{p^k}$-fixed space mapping to a $C_n$-space lands in the $C_{p^k}$-fixed set. Note that as a $C_n$-space, $EC_m\times EC_n \simeq EC_n$. Therefore the result follows.  
 \end{proof}



From Proposition \ref{eqEG}, we readily deduce the following theorem. This is useful in constructing $C_n$-equivariant maps from high-dimensional representation spheres to low-dimensional ones in the following section.
\begin{thm}\label{main} 
If $n$  is  not a prime power and $d$ any integer $\geq 1$, there exists a $C_n$-equivariant map $EC_n \rightarrow  S(\overline{\rho}^{ d})$.
\end{thm}

\begin{proof}
Note that if $n$ is odd or if $n$ is even, $n=p^k\cdot m$ with either $(p,k)\neq (2,1)$ or $m\neq 2$, $\xi^{p^k}\oplus \xi^m$ is a summand of $\bar{\rho}$. Recall $W=S(\xi^{p^k}\oplus \xi^m)$ as in Proposition \ref{eqEG}. Thus there is an equivariant map $W \to S(\bar{\rho}) \subset S(\bar{\rho}^d)$. Therefore in these cases the result follows from Proposition \ref{eqEG}.

Otherwise, we have $n=2l$ with $l$ odd. In this case we consider $W= S(\xi^2 \oplus \sigma)$. As $C_m$ still acts via maps of degree $1$, $W$ satisfies the conditions of Lemma \ref{fibm}. We repeat the proof of Proposition \ref{eqEG} with $p^k$ replaced by $l$ and $m$ by $2$ to obtain an equivariant map $EC_n \to W$. Again $\xi^2 \oplus \sigma$ is a summand of $\bar{\rho}$. Therefore, the result follows. 
\end{proof}

The restriction of the standard representation $W$ of $\Sigma_n$ to $C_n$ is the reduced regular representation $\bar{\rho}$ which is the sum of all the non-trivial irreducible representations of $C_n$ as listed above. 
Restricted to the group $C_n$, the equivariant existence problem (\ref{Tv}) specializes to the following existence problem for $C_n$-equivariant maps.   
$$C_n^{\ast N+1} \to S({\bar{\rho}}^d)$$
with evident $C_n$-action on both sides. We readily observe

\begin{cor} \label{tver}
If $n$  is  not a prime power, there exist $C_n$-equivariant maps from $C_n^{\ast N} \rightarrow S(\overline{\rho}^{ d})$ for all $d\geq 1$ and $N\geq 1$.
\end{cor}

\begin{proof}
As $C_n^{\ast N}$ is a free $C_n$-space, there is an equivariant map $C_n^{\ast N}\to EC_n$. Hence the result follows from Theorem \ref{main}. 
\end{proof}

\subsection{Equivariant maps for a product of elementary Abelian groups}

We take $G=L_n$ the product of elementary Abelian groups $(C_{p_1})^{a_1}\times \cdots \times (C_{p_k})^{a_k}$ where the primes $p_i$ and the integers $a_i$ are defined by the factorization  $n=p_1^{a_1}\cdots p_k^{a_k}$. As $\#(L_n)=n$, the action of $L_n$ on itself by left multiplication induces an injection $L_n\to \Sigma_n$. These groups were used in \cite{Oz87} and \cite{sar} for the proof of the Tverberg's theorem in the prime power case. We also prove that the equivariant existence problem (\ref{Tv}) has a positive answer when $n$ is divisible by at least two distinct primes.  

Note that via the inclusion $L_n \to \Sigma_n$, the standard representation $W$ restricts to the reduced regular representation $\bar{\rho}$ of $L_n$. The following is an analogue of Theorem \ref{main} for the case $G=L_n$. 
\begin{prop} \label{mainln}
If $n$ is not a prime power, there is a $L_n$-equivariant maps $EL_n \to S(\bar{\rho})$, and hence $EL_n \to S(\bar{\rho}^d)$ for any $d$. 
\end{prop}

\begin{proof}
First assume $n$ is odd. The factorization $n=p_1^{a_1}\cdots p_k^{a_k}$ must satisfy $k\geq 2$ by the given condition. Define $m=p_2^{a_2}\cdots p_k^{a_k}$ and note that $L_n = (C_{p_1})^{a_1} \times L_m$.  Fix a projection $L_n \to C_{p_1}$ and another projection $L_n\to C_{p_2}$. Fix a primitive $p_1^{th}$ (respectively, $p_2^{th}$) root of $1$ and let $\xi_1$ (respectively, $\xi_2$) be the induced $L_n$ representation. Let $W = S(\xi_1 + \xi_2)$. We observe first that there is a $L_n$-equivariant map from $EL_n \to W$. 

We follow the steps in Proposition \ref{eqEG}. As, $W^{(C_{p_1})^{a_1}}= S(\xi_2)$ is non-empty, we fix a point $x\in S(\xi_2)$ which induces a $(C_{p_1})^{a_1}$ equivariant map $EL_n \to W$.   
Consider the lifting diagram of $L_m$-spaces
$$\xymatrix{\Map^{(C_{p_1})^{a_1}}(EL_n , W)\ar[r]  &  \Map(EL_n , W) \\
                                                                            &  EL_m \ar[u]^u \ar@{-->}[lu] }$$
Note that $W^{L_m}=S(\xi_1) \neq \varnothing$. Thus, there exists a $L_m$-equivariant constant map $ u : EL_m \rightarrow \Map(EL_n , W)$. Apply Proposition \ref{eqobs} to observe that the obstructions lie in $H^k(BL_m; M_i)$ for $k\geq 1$ where $M_i$ are $p_1$-profinite by Lemma \ref{fibm}. As the cohomology of the group $L_m$ consists of $p_i$-torsion for $i\geq 2$ these groups are $0$. Therefore the lift exists.  A similar argument as the  proof of Proposition \ref{eqEG} using the exponential correspondence  yields a $L_n$-equivariant map $EL_n \to W$. 

Now $\xi_1$ and $\xi_2$ are distinct  non-trivial irreducible representations and hence $\xi_1 + \xi_2$ is a subrepresentation of $\bar{\rho}$. Hence $W$ is a $L_n$-subspace of $S(\bar{\rho})$ and the result follows. For $n$ even, the above proof may be repeated by assuming $p_1$ to be odd and that  $\xi_2$ is induced by the sign representation if $p_2=2$.  
\end{proof}

Note that for $L_n \subset \Sigma_n$ the equivariant problem (\ref{Tv}) becomes the following existence question 
  $$L_n^{\ast N+1} \to S({\bar{\rho}}^d)$$
As the action of $L_n$ on $L_n^{\ast N+1}$ is free, we readily observe 
\begin{cor} 
If $n$  is  not a prime power, there exist $L_n$-equivariant maps from $L_n^{\ast N} \rightarrow S(\overline{\rho}^d)$ for all $d\geq 1$ and $N\geq 1$.
\end{cor}

\subsection{Equivariant maps for dihedral groups}\label{tv2}
Let  $D_n\cong C_n \rtimes C_2$ denote the dihedral group of order $2n$ : $\langle x,y | x^n = y^2 =1 , yx = x^{-1}y \rangle$. We consider the case when $n$ is odd. The real irreducible representations of $D_n$ may be denoted by 
 $\{ 1, \sigma, \hat{\xi}, \hat{\xi}^2, \cdots, \hat{\xi}^{\frac{n-1}{2}}\}$. The representation $\sigma$ is given by the sign representation of $D_n/C_n\cong C_2$. The representation $\hat{\xi}^r$ is the unique real representation whose restriction to  the group $C_n$ is $\xi^r$. The elements of order $2$ in $D_n$ act on $\hat{\xi}^r$ by reflections.

One may apply the techniques of \cite{Dror} as in section \ref{tv} to obtain equivariant maps from $ED_n$ to representation spheres. 
\begin{prop}\label{edn}
Write $n=p^k\cdot m$ where $p\nmid m$. Then there exist $D_n$-equivariant maps from $ED_n$ to $S(\hat{\xi}^{p^k} + \hat{\xi}^m)$.
\end{prop}

\begin{proof}
Write $V= \hat{\xi}^{p^k} + \hat{\xi}^m$. From Proposition \ref{eqEG} we have a $C_n$-equivariant map $EC_n  \to S(V)$.  As $S(V)^{C_2} \neq \varnothing$, we have a $C_2$-equivariant map $EC_2 \to \Map(ED_n,S(V))$ which lifts to $\Map^{C_n}(ED_n,S(V))$ by obstruction theory and Lemma \ref{fibm}. Thus there exists a $D_n$-equivariant map $ED_n \to S(V)$.
\end{proof}

Fix $n$ odd which is not a prime power. For the dihedral group $D_n$, we are interested in maps of the kind 
\begin{equation}\label{tvdn}
\{1,2,\cdots, n\}^{\ast N+1} \to S({\hat{\rho}}^d)
\end{equation} 
The representation $\hat{\rho}$ is the restriction of the standard representation $W$ of $\Sigma_n$ to $D_n$. It follows that as a direct sum of irreducible representations, $\hat{\rho}$ is the sum $\sum \hat{\xi}^r$. Proposition \ref{edn} implies that there are equivariant maps $ED_n \to S(\hat{\rho})$ and hence to $S(\hat{\rho}^d)$ for any $d$. However this does not produce equivariant maps as in (\ref{tvdn}) because the action of $D_n$ on $\{1,\cdots,n\}$ is not free. It turns out (as we observe below) that one needs to use universal spaces of a certain family of subgroups instead of $ED_n$.  

 The action of $D_n$ on the set $\{1,2,\cdots, n\}$ is equivalent to the action on the normal subgroup $\langle x\rangle$. As a $D_n$-set, this is isomorphic to $D_n/H$ where $H=\langle y\rangle$. Note that the assumption $n$ odd implies that all the elements of order $2$ in $D_n$ are conjugate. 

Define $E_HD_n = (D_n/H)^{ \ast {\infty}},$ the infinite join of $D_n/H.$ Note that this satisfies 
$$E_HD_n^K = \begin{cases}\simeq \ast , \text{if $K=e$ or $K$ has order $2$ } \\ 
\varnothing, \text{otherwise} \end{cases}$$
Hence $E_HD_n$ is the universal space for the family of subgroups given by subconjugates of $H$. We  compute the Bredon cohomology of $E_HD_n$. Let $\DD$ denote the orbit category of $D_n$ and $\MM$ a coefficient system (that is, a contravariant functor from $\DD$ to abelian groups). Denote by $\MM^e$ the $D_n$-module $\MM(D_n/e)$. We have the Proposition

\begin{prop}\label{cohehdn}
There is a long exact sequence in the range $i\geq 2$
$$\cdots \to H^i_{D_n}(E_HD_n;\MM) \to H^i(D_n;\MM^e) \to H^i(H;\MM^e) \to H^{i+1}_{D_n}(E_HD_n;\MM) \to \cdots.$$
For $i=1$ there is an abelian group $A$ such that the above long exact sequence extends to the $i=0$ case with $H^1_{D_n}(E_HD_n;\MM)$ replaced by $A$ and there is a surjective map 
$$A \to H^1_{D_n}(E_HD_n;\MM).$$
\end{prop} 

\begin{proof}
We recall from Proposition 4.8 of \cite{May} that $H^i_{D_n}(E_HD_n;\MM) \cong Ext^i_\DD ( \Z_H, \MM)$ where $\Z_H$ is the coefficient which is $\Z$ at sub-conjugates of $H$ and $0$ otherwise. We map the projective $F_{D_n/H}$ (\cite{bredon}, part I, pg 23) to $\Z_H$. Note 
$$F_{D_n/H}(G/K)= \begin{cases}\Z , \text{if $K=H$ or a conjugate of $H$ } \\
\Z\{D_n/H\}, \text{if $K=e$}\\ 
0, \text{otherwise} \end{cases}$$
We have a surjective map $F_{D_n/H} \to \Z_H$. The kernel is $0$ at $D_n/K$ for $K\neq e$ and the kernel at $D_n/e$ is $\KK=Ker( \Z\{D_n/H\} \to \Z)$. Therefore for $i\geq 2$, 
$$H^i_{D_n}(E_HD_n;\MM) \cong Ext^i_\DD ( \Z_H, \MM) \cong Ext^{i-1}_{D_n}(\KK, \MM^e)$$
From Proposition 16.1 of \cite{hilt-stam}, the short exact sequence $\KK \to \Z\{D_n/H\} \to \Z$ of $D_n$-modules induces a long exact sequence 
$$ \cdots \to Ext^{i-1}_{D_n}(\KK,\MM^e) \to H^i(D_n;\MM^e) \to H^i(H;\MM^e) \to Ext^i_{D_n}(\KK,\MM^e) \to \cdots$$
The first statement follows. For the second, put $A=Ext^0_{D_n}(\KK,\MM^e)$. The short exact sequence $0\to \KK \to F_{D_n/H} \to \Z_H\to 0$ induces an exact sequence
$$Hom_\DD (F_{D_n/H},\MM) \to Ext^0_{D_n}(\KK,\MM^e)\to Ext^1_\DD (\Z_H,\MM)\to 0$$
The second statement follows.
\end{proof}

From Proposition \ref{cohehdn}, we readily deduce
\begin{cor}\label{obsgpdn}
Suppose that $\MM$ is a coefficient system such that $\MM^e$ is $m$-profinite for $m$ relatively prime to $2n$. Then for $i>1$, $H^i_{D_n}(E_HD_n; \MM) =0$. 
\end{cor}

\begin{proof}
We have from the given conditions that $H^k(D_n;\MM^e)=0$ and $H^k(H;\MM^e)=0$ if $k>0$. Thus the long exact sequence in Proposition \ref{cohehdn} implies $H^i_{D_n}(E_HD_n; \MM) =0$ for $i\geq 2$. 
%
\end{proof}

Now we may work out the obstruction theory for $E_HD_n$ as before to deduce the existence of maps as in (\ref{tvdn}). Write $n=p^k\cdot m$ and $W=S(\hat{\xi}^{p^k} + \hat{\xi}^m)$.
 \begin{prop}\label{eqEG2}
There exists a $D_{n}$-map from $E_H D_n \rightarrow {W}$.
\end{prop}

\begin{proof}
We write $D_n = C_{p^k} \rtimes D_m$. Note that $S(\hat{\xi}^{p^k})^{C_{p^k}} \neq \varnothing$ and $S(\hat{\xi}^m)^{D_m} \neq \varnothing$. So we have a $C_{p^k}$-equivariant map $E_H D_n \to W$ and one considers the lifting diagram
$$\xymatrix{\Map^{C_{p^k}}(E_H D_n , W)\ar[r]^i  &  \Map(E_H D_n , W) \\
                                                                            &  E_HD_m \ar[u]^u \ar@{-->}[lu] }$$
The first step involves lifting the map over the $0$-skeleton. For this we require a $D_m$-equivariant map $D_m/H\to \Map^{C_{p^k}}(E_H D_n , W)$ which is homotopic to the restriction of $u$. We take a $C_{p^k}$-equivariant constant map $v$ which maps $E_H D_n$ to $S(\hat{\xi}^{p^k})^H\simeq S^0$. Thus $v\in \Map^{C_{p^k}}(E_HD_n , W)^H$, and hence $v$ produces a $D_m$-map  $D_m/H\to \Map^{C_{p^k}}(E_H D_n , W)$. Now the composition onto $\Map(E_H D_n , W)$ maps $D_m/H$ to the corresponding constant map to a point in $S(\hat{\xi}^{p^k})^H$. In $W=S(\hat{\xi}^{p^k} + \hat{\xi}^m)=S(\hat{\xi}^{p^k})\ast S( \hat{\xi}^m)$ this is $D_m$-homotopic to the constant map onto $S( \hat{\xi}^m)^H$. 

The obstructions to lifting the homotopy class $u$ starting at $v$ lie in the groups $H^j_{D_m}(E_HD_m; \underline{\pi_{j-1}F(f)})$ by Proposition \ref{eqobs}. From Lemma \ref{fibm} the groups $\underline{\pi_{j-1}F(f)}^e$ are $p$-profinite and hence $H^j_{D_m}(E_H D_m; \underline{\pi_{j-1}F(f)})=0$ if $j\geq 2$ by Corollary \ref{obsgpdn}. Therefore the obstructions are forced to be $0$ if $j\ge 2$. 

For $j=1$ we apply Proposition \ref{cohehdn}. Let $\MM$ denote $ \underline{\pi_0F(f)}$. Then we have a surjective map $A\to H^1_{D_m}(E_H D_m;\MM)$ and an exact sequence 
$$H^0(D_m;\MM^e)\to H^0(H;\MM^e) \to A \to H^1(D_m; \MM^e) \to H^1(H;\MM^e) $$
From Proposition \ref{fibhtpy} we have $\MM^e=0$. Thus $A=0$ and hence, $H^1_{D_m}(E_HD_m;\MM)=0$.  Thus we have
$$ \Map^{D_m}(E_HD_m, \Map^{C_{p^k}}(E_H D_n, W))\neq \varnothing$$
which gives a $D_n$-equivariant map $E_HD_m\times E_HD_n \to W$. Since $E_HD_m \times  E_H D_n \simeq  E_H D_n$  is a $D_n$-homotopy equivalence, we get a $D_n$-map from $ E_H D_{n} \rightarrow W.$ 
\end{proof} 

Proposition \ref{eqEG2} implies the following result as in Theorem \ref{main}. This also allows us to construct $D_n$-equivariant maps from high dimensional representation spheres to lower dimensional ones in the following section. 
\begin{thm}\label{main2}
If $n$  is odd and not a prime power, then there exists a $D_{n}$-equivariant map $E_H D_n \rightarrow  S(\hat{\rho}^{d})$ for all $d.$ 
\end{thm}
The following corollary readily follows encoding the existence of the maps in (\ref{tvdn}).
  \begin{cor}\label{tver2}
 If $n$  is odd and not a prime power, then there exists a $D_n$-equivariant map $ (D_n/H)^{\ast k} \to  S(\hat{\rho}^{ d})$ for any positive integers $k$ and $d$.
 \end{cor}

\begin{proof}
Note that the isotropy groups of points in $(D_n/H)^{\ast k}$ are contained in sub-conjugates of $H$. Thus there exists a $D_n$-equivariant map $(D_n/H)^{\ast k}\to E_HD_n$. The result follows from Theorem \ref{main2}.
\end{proof}

\section{Applications}
In this section we discuss some applications of the existence of equivariant maps proved in Theorems \ref{main} and \ref{main2}. Most of these applications are usually considered alongside the topological Tverberg theorem. 

We start with an easy consequence of Theorem \ref{main}. Recall theorem 2.6 of  \cite{kriz} which asserts the implication $(1)$ $\implies$ $(2)$ $\implies$ $(3)$ of three statements $(1),(2),(3)$. The implication $(2) \implies (3)$ asserts : If $C$ is a finite $C_n$ complex with a $C_n$-equivariant map $C\to S(\bar{\rho}^d)$ then there is an $i<d(n-1)$ such that $H^i(C)\neq 0$. It was found to be erroneous in \cite{krizcorr} and it is already contradicted by \cite{Oz87}. From Theorem \ref{main} we find that if $n$ is not a prime power, for every $k$ and $d$ there is a finite $C_n$-complex $C$ with $H^i(C)=0$ for $i\leq k$ and there exists a $C_n$-equivariant map $C\to S(\bar{\rho}^d)$. For example we may choose $C$ to be an arbitrarily large skeleton of $EC_n$. This rules out the existence of weaker versions of theorem 2.6 of \cite{kriz}. 

Other arguments that are commonly used in the proof of the topological Tverberg theorem in the prime power case are generalizations of the Borsuk-Ulam theorem. There are many definitions of Borsuk-Ulam theorems in the literature. One category of definitions are of Borsuk-Ulam properties of groups (see \cite{mar90}) which encodes whether equivariant maps between fixed point free representation spheres implies non-zero degree. The other approach is to define Borsuk-Ulam properties of representations for a fixed group $G$.  In this respect we find two different definitions \\
(1) A fixed point free $C_n$-representation $V$ of dimension $N$ has Borsuk-Ulam property if any equivariant map from $C_n^{\ast N+1}\to V$ has a zero. (\cite{sar})\\
(2) A $G$-representation $V$ is said to have the Borsuk-Ulam property if the existence of a $G$-map from $S(W) \to S(V)$ implies $\dim(W)\leq \dim(V)$. (\cite{iz-ma})

We apply the techniques in section \ref{tv} to the above situation.  
\begin{theorem}\label{borcycsqfree}
Suppose that $n$ is square free. Let $V$ be a fixed point free representation of $C_n$. Then either $V$ has the Borsuk-Ulam properties ((1) and (2)) as above or there exists a $C_n$-equivariant map $EC_n\to S(V)$.  
\end{theorem}  

\begin{proof}
If $n$ is prime, a fixed point free representation is necessarily free. Such representations have the Borsuk-Ulam property (\cite{dold83},\cite{liu83}). We restrict our attention to $n$ square free and not prime so that $n=p_1\cdots p_k$ for $k>1$. 

First  consider the case where $V$ is the realization of a complex representation. Then $V = \oplus_i m_i\xi^i$. From \cite{sar} and \cite{iz-ma} we have that $V$ has the Borsuk-Ulam property if the Euler class $e(V)=\prod_i i^{m_i}$ is not divisible by $n$. Otherwise $n\mid \prod_i i^{m_i}$ which happens if and only if for every $j\leq k$, there is an $i$ with $m_i\neq 0$ so that  $p_j \mid i$. If all the $p_j$ divide the same $i$ then $V$ contains the representation $\xi^n = \epsilon_\C$ the trivial representation and hence is not fixed point free. Therefore, $V$ contains a summand $\xi^r + V_1$ so that $p_1\cdots p_s \mid r$ and the restriction of $V_1$ to the group $C_{p_{s+1}\cdots p_k}$ has Euler class $0$ (after a possible rearrangement of $p_j$).

We may proceed by induction on $k$ the number of distinct primes dividing $n$. Write $n=l\cdot m$ where $l= p_1\cdots p_s$ and $m= p_{s+1}\cdots p_k$ so that $C_n\cong C_l\times C_m$. From the induction hypothesis, there exists an equivariant map $u:EC_m \to S(V_1)$. We may now proceed as in Proposition \ref{eqEG}. Since $\xi^r$ is a $C_l$-fixed it follows that $\Map^{C_l}(EC_n, S(V))$ is non-empty. Also $S(V)$ satisfies the conditions of Lemma \ref{fibm} as a $C_l$-space. Now we may consider the appropriate lifting diagram  
$$\xymatrix{\Map^{C_l}(EC_n , S(V))\ar[r]  &  \Map(EC_n , S(V)) \\
                                                                            &  EC_m \ar[u]^u \ar@{-->}[lu] }$$
whose obstructions lie in $H^*(C_m;F)$ where $F$ is $l$-profinite by Lemma \ref{fibm} and hence are $0$.  Thus we obtain a $C_n$ equivariant map $EC_n \to S(V)$ as required. 

It remains to consider the case where $V$ has a one-dimensional real irreducible summand. In this case $n$ is even $=2l$ and $V= \sigma + V_1$ where $V_1$ is the realization of a complex representation. If $e(V_1)=0$ we are done by the above argument. 

We have $e(V) \in H^{\dim(V_1)+1}(C_{2l};\hat{\Z})$. The local coefficient $\hat{\Z}$ has the action corresponding to orientability of the associated representation; thus, it is a tensor product of the trivial coefficient $\Z$ over $C_l$ and the usual action of $\pm 1$ over $C_2$. The class $e(V)$ is the tensor product of $e(V_1)\in H^{\dim(V_1)}(C_l;\Z)$ and the class $e(\sigma)\in H^1(C_2; \hat{\Z})$. The class $e(\sigma)$ can be identified as the characteristic class of the canonical line bundle over $\R P^\infty$ which is the unique non-zero class in $H^1(\R P^\infty; \hat{\Z})\cong \Z/2$. Therefore $e(V)=0$ if and only if $e(V_1) \equiv 0 (mod~2)$. 

Let $V=V_1 \oplus \sigma$ such that $e(V)=0$ so that $V$ does not have the Borsuk-Ulam property. Then $e(V_1)$ is even which implies $V_1^{C_2} \neq \varnothing$. Note that $\sigma$ is fixed by $C_l$ so that   $\Map^{C_l}(EC_n, S(V))$ is non-empty. As above consider the lifting diagram  
$$\xymatrix{\Map^{C_l}(EC_n , S(V))\ar[r]  &  \Map(EC_n , S(V)) \\
                                                                            &  EC_2 \ar[u]^u \ar@{-->}[lu] }$$
The map $u$ is obtained as $\varnothing \neq S(V_1)^{C_2} \subset S(V)^{C_2}$. The obstructions are $0$ in the same way as above.  Thus there exist $C_n$-equivariant maps $EC_n \to S(V)$ as required. 
\end{proof}

The Theorem \ref{borcycsqfree} invites the following definition
 \begin{defn} A real representation $V$ of a finite group $G$ is said to have the anti-Borsuk-Ulam property  if for every $r$  there is a representation $U$ with $dim(U) > r$ and a continuous $G$- equivariant map from $S(U)$ to $S(V)$.
 \end{defn}

As $S(r\xi)$ is free, there is a map $S(r\xi) \to EC_n$. From Theorem \ref{borcycsqfree} we conclude that if $n$ is square free, either fixed point free representations have the Borsuk-Ulam property or the anti-Borsuk-Ulam property. However, this does not hold if $n$ is not square free as demonstrated in the following example. 

\begin{exam}
Let $G=C_{p^2}$ and $V=2\xi^p$ for a prime $p$. Then $e(V)=0$ so that $V$ does not have the Borsuk-Ulam property. However $G$ is a $p$-group and $V^G = (0)$. Thus by Sullivan's conjecture (\cite{May},\cite{Mill}) the homotopy fixed points of $S(V)$ are empty. Hence there is no $C_{p^2}$-equivariant map $EC_{p^2} \to S(V)$. Thus we cannot deduce the anti-Borsuk-Ulam property for $V$ in the above manner. 
\end{exam}

One can easily generate similar examples to the above for $G= C_{p^k}$. We observe below that these are the only kind of examples in which the dichotomy of the Borsuk-Ulam property and the anti-Borsuk-Ulam property can fail for $G=C_n$. Note that if for every prime power $p^k$ dividing $n$, $V^{C_{p^k}}\neq (0)$, then $e(V)=0$. From Sullivan's conjecture we know that if $V^{C_{p^k}}=(0)$ for $C_{p^k}\subset C_n$ there cannot be a $C_n$-equivariant map $EC_n \to S(V)$.   
\begin{theorem}\label{borcyc}
Let $V$ be a fixed point free $C_n$-representation such that for each prime power $p^k$ dividing $n$, $V^{C_{p^k}} \neq (0)$. Then $V$ has the anti-Borsuk-Ulam property. 
\end{theorem}

\begin{proof}
We proceed by induction on the number of primes dividing $n$. The starting point is at a prime power $n=p^k$  whence the condition implies that $S(V)$ has a fixed point. In general,  write $n=p^k\cdot m$ so that the number of primes dividing $m$ is less than $n$. We have a $C_m$-equivariant map $u:EC_m \to S(V)$ by induction hypothesis and a $C_{p^k}$-equivariant map $EC_n \to S(V)$. We may assume $p\neq 2$ so that the group elements in $C_{p^k}$ act by degree $1$ maps. We may now proceed as in Proposition \ref{eqEG} via the lifting diagram  
$$\xymatrix{\Map^{C_{p^k}}(EC_n , S(V))\ar[r]  &  \Map(EC_n , S(V)) \\
                                                                            &  EC_m \ar[u]^u \ar@{-->}[lu] }$$
whose obstructions are $0$.  Thus we obtain a $C_n$-equivariant map $EC_n \to S(V)$ as required. 
\end{proof}

We now consider the case of the dihedral group $D_n$ for $n$ odd. Again we have similar results in the square free case and the general non prime power case.
\begin{theorem}\label{bordnsqfree}
Suppose that $n$ is odd and square free and let $V$ be a fixed point free $D_n$-representation not containing $\sigma$. Then either $V$ has the Borsuk-Ulam property or the anti-Borsuk-Ulam property. 
\end{theorem}  

\begin{proof}
If $n$ is prime,  $V$ is a sum of representations $\hat{\xi}^i$. Then the restriction of $V$ to $C_n$ has the Borsuk-Ulam property from Theorem \ref{borcycsqfree} implying $V$ also carries the property. 
Now consider $n$ square free and not prime so that $n=p_1\cdots p_k$ for $k>1$. 

We have $V = \oplus_i m_i\hat{\xi}^i$. By restriction to $C_n$ observe $V$ has the Borsuk-Ulam property if $\prod_i i^{m_i}$ is not divisible by $n$. From the proof of Theorem \ref{borcycsqfree}, $V$ contains a summand $\hat{\xi}^r + V_1$ so that $p_1\cdots p_s \mid r$ and the restriction of $V_1$ to the group $C_{p_{s+1}\cdots p_k}$ has Euler class $0$. We proceed by induction on $k$, the number of prime factors, to prove that there is a map $E_HD_n \to S(V)$. 

Write $n=l\cdot m$ where $l= p_1\cdots p_s$ and $m= p_{s+1}\cdots p_k$ so that $D_n\cong C_l\rtimes D_m$. From the induction hypothesis, there exists an equivariant map $u:E_HD_m \to S(V_1)$. Now proceed as in Proposition \ref{eqEG2}. Since $\hat{\xi}^r$ is a $C_l$-fixed it follows that $\Map^{C_l}(E_HD_n, S(V))$ is non-empty. Also $S(V)$ satisfies the conditions of Lemma \ref{fibm} as a $C_l$-space. Now we may consider the appropriate lifting diagram  
$$\xymatrix{\Map^{C_l}(E_HD_n , S(V))\ar[r]  &  \Map(E_HD_n , S(V)) \\
                                                                            &  E_HD_m \ar[u]^u \ar@{-->}[lu] }$$
and check that the obstructions are $0$ as in Proposition \ref{eqEG2}. 

For the remaining note that there is a $D_n$-equivariant map $S(k\hat{\xi}) \to E_HD_n$ for any $k$. Thus we obtain maps from representations of arbitrarily large dimension.
\end{proof}

In the same manner, we obtain an analogue of Theorem \ref{borcyc}.
\begin{theorem}\label{bordn}
Let $V$ be a fixed point free $D_n$-representation not containing $\sigma$ such that for each prime power $p^k$ dividing $n$, $V^{C_{p^k}} \neq (0)$. Then $V$ has the anti-Borsuk-Ulam property. 
\end{theorem}

\begin{proof}
We proceed by induction on the number of primes dividing $n$. The starting point is at a prime power $n=p^k$  whence the condition implies that $S(V)$ has a fixed point. In general,  write $n=p^k\cdot m$ so that the number of primes dividing $m$ is less than $n$. We have a $D_m$-equivariant map $u:E_HD_m \to S(V)$ by induction hypothesis and a $C_{p^k}$-equivariant map $E_HD_n \to S(V)$. We may assume $p\neq 2$ so that the group elements in $C_{p^k}$ act by degree $1$ maps. We may now proceed as in Proposition \ref{eqEG2} and lift the map from $E_HD_m \to \Map(E_H D_n,S(V))$ to $E_HD_m \to \Map^{C_{p^k}}(E_HD_n,S(V))$. This gives us a $D_n$-map $E_HD_n\to S(V)$.
\end{proof}

\mbox{ }\\

\end{document}